\documentclass[12pt,a4paper]{article}
\usepackage{geometry}
\usepackage{hyperref}

\usepackage{times,authblk}
\usepackage[UKenglish]{babel}
\renewcommand{\phi}{\varphi}
\renewcommand{\rho}{\varrho}
\renewcommand{\epsilon}{\varepsilon}
\usepackage[cal=esstix]{mathalfa}  
\usepackage{amsmath, amsfonts, amssymb, amsthm, bm, stmaryrd, enumerate, mathtools}
\usepackage{graphicx,tikz,tikz-cd}
\tikzset{commutative diagrams/.cd}

\newtheorem{Def}{Definition}[section]
\newenvironment{definition}{\begin{Def} \rm}{\end{Def}}
\newtheorem{lemma}[Def]{Lemma}
\newtheorem{proposition}[Def]{Proposition}
\newtheorem{theorem}[Def]{Theorem}
\newtheorem{example}[Def]{Example}

\newcommand{\komma}{,\hspace{0.3em}}
\newcommand{\id}{\text{\rm id}}
\renewcommand{\leq}{\leqslant}
\renewcommand{\geq}{\geqslant}

\newcommand{\Naturals}{{\mathbb N}}
\newcommand{\Reals}{{\mathbb R}}
\newcommand{\Complexes}{{\mathbb C}}
\newcommand{\Quaternions}{{\mathbb H}}
\newcommand{\inv}{^\star}

\renewcommand{\c}{^\perp}
\newcommand{\herm}[2]{\left( #1 \,|\, #2 \right)}
\newcommand{\adj}{^\star}
\newcommand{\Hil}[1]{\mathcal{Hil}_{#1}}
\newcommand{\Her}[1]{\mathcal{Her}_{#1}}
\newcommand{\C}{{\mathcal C}}
\newcommand{\Cdm}{{\mathcal C}_\text{\rm dm}}
\newcommand{\E}{{\mathsf E}}
\newcommand{\op}{^\text{op}}
\DeclareMathOperator{\kernel}{ker}
\DeclareMathOperator{\image}{im}
\DeclareMathOperator{\homset}{Hom}

\begin{document}

\title{Axiomatising the dagger category \\ of complex Hilbert spaces}

\author[1]{Jan Paseka}

\author[2]{Thomas Vetterlein}

\affil[1]{\footnotesize Department of Mathematics and Statistics,
Masaryk University \authorcr
Kotl\'a\v rsk\'a 2, 611\,37 Brno, Czech Republic \authorcr
{\tt paseka@math.muni.cz}}

\affil[2]{\footnotesize Institute for Mathematical Methods in Medicine and Data Based Modeling, \authorcr
Johannes Kepler University Linz \authorcr
Altenberger Stra\ss{}e 69, 4040 Linz, Austria \authorcr
{\tt Thomas.Vetterlein@jku.at}}

\date{\today}

\maketitle

\begin{abstract}\parindent0pt\parskip1ex

\noindent We axiomatise the dagger category of complex Hilbert spaces and bounded linear maps, using exclusively purely categorical conditions. Our axioms are chosen with the aim of an easy interpretability: two of them describe the composition of objects, two further ones deal with the decomposition of objects, and a final axiom expresses a symmetry property.

The categorical reconstruction of complex Hilbert spaces addresses foundational issues in quantum physics. We present a simplified alternative to recent characterisations.



\mbox{}\vspace{-2ex}

\end{abstract}

\section{Introduction}
\label{sec:Introduction}

The complex Hilbert space plays a prominent role in quantum physics and in numerous branches of mathematics, most notably in harmonic analysis. However, it is not straightforward to explain precisely why this mathematical model is so useful and versatile. Therefore, it seems sensible to look for ways to justify the significance of this model. One might ask, in particular, whether the Hilbert space can be reconstructed from a framework that is easier to specify and to understand. We note that a Hilbert space may be recovered from the ortholattice of its closed subspaces \cite{Wil}, or from the orthoset of its one-dimensional subspaces \cite{Vet1,Vet2,Rum}. However, reducing the analytical model to a first-order structure does not necessarily lead to simpler axioms. The conditions are usually technically involved and difficult to motivate.

In this paper, we take up this issue and we adopt the categorical approach to the foundations of quantum mechanics, which was pioneered by Samson Abramsky and Bob Coecke \cite{AbCo}. Rather than considering single Hilbert spaces in isolation and describing their internal structure, the idea is to focus on their mutual relationship and disregard their inner structure. Alongside the standard categorical concepts, a further ingredient is added: the dagger \cite{AbCo,Sel}. The strength of this approach was demonstrated in the celebrated work of Chris Heunen and Andre Kornell, who succeeded in characterising Hilbert spaces in purely categorical terms \cite{HeKo}.

However, the physical interpretability of the postulated categorical properties has so far remained an unsolved issue. Elaborating on this point seems reasonable and we propose to this end an alternative axiomatic approach. We present in this paper a characterisation of the dagger category of complex Hilbert spaces and bounded linear maps between them, the dagger being the usual adjoint. The most notable difference with \cite{HeKo} is that we do not assume a monoidal structure, the division ring being determined in an alternative way. Instead, we assume the presence of a dagger simple object $I$ and the division ring is constructed from the endomorphisms of $I$.

Let us outline our hypotheses and indicate how they are meant to be understood. For the precise formulation, see the beginning of Section~\ref{sec:main}. We refer in what follows to a dagger category with a zero object.  There are two conditions regarding the {\it composition} of objects from given ones: (H1) postulates the existence of biproducts and (H2) implies the existence of colimits of directed diagrams in the wide subcategory of dagger monomorphisms. The next two conditions are related to the {\it decomposition} of objects: by (H3), any dagger monomorphism is part of a biproduct; and (H4) postulates the existence of a certain dagger simple object $I$, i.e., an object every non-zero dagger monomorphism into which is an isomorphism. $I$ is to be understood as the smallest, ``indecomposable'' non-zero object. The final condition concerns the symmetries of an object: according to (H5), any dagger automorphism is supposed to have strict square roots. In terms of Hilbert spaces, the meaning of the hypotheses might be obvious. The biproducts are the orthogonal sums, and the mentioned colimits are direct limits of an isometric direct system. Moreover, any Hilbert space is the orthogonal sum of a closed subspace and its orthocomplement; and the ``indecomposable'' object is of course a one-dimensional Hilbert space. Finally, any unitary operator $U$ in a complex Hilbert space has a strict square root $V$. This means that $V$ is a further unitary operator such that $V^2 = U$ and the reducing subspaces of $U$ and $V$ coincide.

Condition (H1) is part of Heunen and Kornell's axiom system as well, it is their condition (B). The first part of our (H2) is their (C). We do not assume dagger equalisers and (H3) replaces (E) and (K). Moreover, the role of the monoidal identity is taken over by our dagger simple $I$, thus (H4) replaces (T). Finally, we do not require $I$ to be a simple generator; but the second part of our (H2) is a related condition.

While writing this paper, we became aware of a further work that is based on similar ideas as ours. Stephen Lack and Shay Tobin's article \cite{LaTo} contains an axiomatisation of the dagger category of the classical Hilbert spaces and the approach likewise does not require a monoidal structure. A notable difference is that dagger equalisers and dagger kernels are in \cite{LaTo} assumed to exist.

Compared to either of the previous works, our objective is narrower: we actually characterise the Hilbert spaces over $\Complexes$. We rule out $\Reals$ and $\Quaternions$ by means by condition (H5), which consequently has no analogue in \cite{HeKo} or \cite{LaTo}. We note that the limitation to $\Complexes$ allows us to make convenient use of Bikchentaev's Theorem \cite{Bik}, according to which in any at least two-dimensional complex Hilbert space the collection of bounded operators is generated {\it as a ring} by the projections.

The paper is structured as follows. The necessary background information on dagger categories and Hermitian spaces is provided in Sections~\ref{sec:dagger-categories} and~\ref{sec:Hermitian-spaces}, respectively. In Section~\ref{sec:main} we present our list of axioms and we show that a dagger category fulfilling the conditions is unitarily dagger equivalent to the dagger category of complex Hilbert spaces.

\section{Dagger categories}
\label{sec:dagger-categories}

In this preliminary section, we compile basic notions and facts from category theory. For further details, we refer the reader, e.g., to Heunen and Vicary's monograph \cite{HeVi}.

A {\it dagger category} is a category equipped with an involutive contravariant endofunctor that is the identity on objects. Such an endofunctor is called a {\it dagger} and we denote the image of a morphism $f$ under it by $f\adj$.

Let $\C$ be a dagger category. A morphism $f \colon A \to B$ is called a {\it dagger monomorphism} if $f\adj \circ f = \id_A$, and a {\it dagger isomorphism} if $f\adj \circ f = \id_A$ and $f \circ f\adj = \id_B$.

Let us assume that $\C$ has a zero object $0$. We denote the morphism $A \to B$ factorising through $0$ by $0_{A,B}$. Note that ${0_{A,B}}\adj = 0_{B,A}$. Moreover, we call an object $I 
\in \C$ {\it dagger simple} if $I$ is non-zero and any non-zero dagger monomorphism to $I$ is a dagger isomorphism.

By a {\it dagger biproduct} of $A, B \in \C$, we mean a coproduct
\[ \begin{tikzcd} A \arrow[r, "\iota_A"] & A \oplus B & B \arrow[l, "\iota_B"'] \end{tikzcd} \]
such that $\iota_A, \iota_B$ are dagger monomorphisms and ${\iota_B}\adj \circ \iota_A = 0_{A,B}$. We will refer to $\iota_A, \iota_B$ as the {\it injections}, to ${\iota_A}\adj, {\iota_B}\adj$ as the {\it projections} of the dagger biproduct.

Similarly, we define the dagger biproduct of $n$ objects for any finite number $n$. In case when $n = 0$, the biproduct is understood to be $0$. In case when $n = 1$, the injection and the projection is a pair of mutually inverse dagger isomorphisms. The $n$-fold dagger biproduct of some $A \in \C$ is written $n \, A$. Up to dagger isomorphism, the dagger biproduct is associative. Moreover, the zero object behaves neutrally: $\begin{tikzcd}[cramped] A \arrow[r, "\id_A"] & A & 0 \arrow[l, "\;\;0_{0,A}"'] \end{tikzcd}$ is obviously a dagger biproduct. Hence we may assume that $A \oplus 0 = A$.

For the remainder of the section, let us suppose that $\C$ has dagger biproducts. This is to say that any pair, and hence any finite set of objects, has a dagger biproduct.

\begin{lemma} \label{lem:zero-in-biproduct}
Let $\begin{tikzcd}[cramped] A \arrow[r, "f"] & A \oplus B & B \arrow[l, "g"'] \end{tikzcd}$ be a dagger biproduct. If $f = 0$, then $g$ is a dagger isomorphism.
\end{lemma}

\begin{proof}
Assume $f = 0$. By definition of the biproduct, $g\adj \circ g = \id_B$. Moreover, $g \circ g\adj \circ g = g$ and $g \circ g\adj \circ f = 0_{A, A \oplus B}$, hence $g \circ g\adj = \id_{A \oplus B}$.
\end{proof}

For morphisms $f \colon A \to C$ and $g \colon B \to D$, we let $f \oplus g$ be the unique morphism making the diagram
\[ \begin{tikzcd}
C \arrow[r, "\iota_C"] & C \oplus D & D \arrow[l, "\iota_D"'] \\
A \arrow[r, "\iota_A"] \arrow[u, "f"] & A \oplus B  \arrow[u, "f \oplus g"] & B \arrow[l, "\iota_B"'] \arrow[u, "g"]
\end{tikzcd} \]
commutative. Similarly, we define the sum of any $k \geq 1$ morphisms. The sum of morphisms is again associative, and for any $f \colon A \to B$ we have $f \oplus 0_{0,0} = f$.

\begin{lemma} \label{lem:f-oplus-isomorphy}
Let $f_1 \colon A_1 \to B_1$ and $f_2 \colon A_2 \to B_2$ be morphisms. Then $(f_1 \oplus f_2)\adj = {f_1}\adj \oplus {f_2}\adj$.
\end{lemma}

\begin{proof}
This is shown in \cite[Chapter 2]{HeVi}. Alternatively, see \cite[Lemma~5.2]{PaVe} for a short direct proof.
\end{proof}

A {\it semiadditive structure} on $\C$ is a binary operation $+$ on $\homset(X,Y)$ for each $X, Y \in \C$, subject to the following conditions: (i) equipped with $+$ and $0_{X,Y}$, $\homset(X,Y)$ is a commutative monoid, and (ii) the composition of morphisms is left and right distributive over $+$.

We will write $\Delta_X$ for the diagonal morphism of an object $X$, and $\nabla_X$ for the codiagonal morphism of $X$.

\begin{theorem} \label{thm:semiadditive-structure}
$\C$ possesses a unique semiadditive structure, given by
\begin{equation} \label{fml:semiadditive-structure}
f+g \;=\; \nabla_Y \circ (f \oplus g) \circ \Delta_X
\end{equation}
for morphisms $f, g \colon X \to Y$.
\end{theorem}

\begin{proof}
See \cite[Proposition~40.12]{HeSt}.
\end{proof}

For morphisms $f, g \colon X \to Y$ in a category with biproducts, $f+g$ is specified by means of the following diagram, in which every triangle and every square of the following diagram commutes:
\begin{equation} \label{fml:addition}
\begin{tikzcd}
& X \arrow[dl, "\id"', shift right=0.5ex] \arrow[d, dashed, "\Delta_X"] \arrow[dr, "\id", shift left=0.5ex] \\
X \arrow[r, "\iota"', shift right=0.5ex] \arrow[d, "f"]
& X \oplus X \arrow[l, "\iota\adj"', shift right=0.5ex] \arrow[r, "\iota\adj", shift left=0.5ex] \arrow[d, dashed, "f \oplus g"]
& X \arrow[l, "\iota", shift left=0.5ex] \arrow[d, "g"] \\
Y \arrow[r, "\iota"'] \arrow[dr, "\id"', shift right=0.5ex]
& Y \oplus Y \arrow[d, dashed, "\nabla_Y"]
& Y \arrow[l, "\iota"] \arrow[dl, "\id", shift left=0.5ex] \\
& Y
\end{tikzcd}
\end{equation}

\begin{lemma} \label{lem:plus-oplus-2}
Consider the dagger biproduct $\begin{tikzcd}[cramped] A \arrow[r, "\iota_A"] & A \oplus B & B \arrow[l, "\iota_B"'] \end{tikzcd}$. Then
\[ \iota_A \circ {\iota_A}\adj \;+\; \iota_B \circ {\iota_B}\adj \;=\; \id_{A \oplus B}. \]
\end{lemma}

\begin{proof}
Pre-composition with $\iota_A$ leads on both sides to $\iota_A$, and similarly for $\iota_B$.
\end{proof}

\begin{lemma} \label{lem:dagger-of-sum}
For any morphisms $f, g \colon X \to Y$, we have $(f+g)\adj = f\adj + g\adj$.
\end{lemma}

\begin{proof}
We have
\[ \begin{split} (f+g)\adj & \;=\; (\nabla_Y \circ (f \oplus g) \circ \Delta_X)\adj
\;=\; {\Delta_X}\adj \circ (f \oplus g)\adj \circ {\nabla_Y}\adj \\
& \;=\; \nabla_X \circ (f\adj \oplus g\adj) \circ \Delta_Y
\;=\; f\adj + g\adj \end{split} \]
by Lemma \ref{lem:f-oplus-isomorphy}.
\end{proof}

The following additional types of morphisms will be considered. By a {\it projection}, we mean a selfadjoint, idempotent endomorphism. Let $f \colon X \to X$ be a dagger automorphism. Then we call a further dagger automorphism $g \colon X \to X$ a {\it strict square root} of $f$ if $g^2 = f$ and for any projection $p$ of $X$, $p \circ f = f \circ p$ if and only if $p \circ g = g \circ p$.

Finally, the common notions around functors between categories are adapted to the present setting in the straightforward manner \cite{HeVi}. Specifically, a {\it dagger functor} is a functor $\mathcal F \colon \C_1 \to \C_2$ between dagger categories that also preserves the dagger. Furthermore, $\mathcal F$ is called {\it dagger essentially surjective} if each object in $\C_2$ is dagger isomorphic to ${\mathcal F}(X)$ for some object $X$ in $\C_1$.

Moreover, a dagger functor $\mathcal F \colon \C_1 \to \C_2$ establishes a {\it unitary dagger equivalence} if $\mathcal F$ is part of an adjoint equivalence of categories such that the unit and counit natural transformations are dagger isomorphisms at every stage \cite{Vic}. This is the case if and only if $\mathcal F$ is full, faithful, and dagger essentially surjective \cite[Lemma~V.1]{Vic}.

\section{Hilbert spaces and their generalisation}
\label{sec:Hermitian-spaces}

In this further preliminary section, we provide the necessary background information on inner-product spaces.

By a {\it $\star$-sfield}, we mean a division ring $\mathbb F$ equipped with an involutive antiautomorphism $\inv \colon \mathbb F \to \mathbb F$. The {\it classical} $\star$-sfields are $\Reals$, $\Complexes$, and $\Quaternions$ equipped with the identity or the standard conjugation, respectively. A {\it Hermitian space} is a linear space $\mathfrak H$ over a $\star$-sfield equipped with a sesquilinear form $\herm{\cdot}{\cdot}$ such that $\herm{u}{v} = \herm{v}{u}\inv$ for any $u, v \in \mathfrak H$ and $\herm{u}{u} = 0$ only if $u = 0$.

Let $\mathfrak H$ be a Hermitian space. We call $u \in \mathfrak H$ a {\it unit vector} if $\herm u u = 1$, and $\mathfrak H$ is called {\it uniform} if each $1$-dimensional subspace of $\mathfrak H$ contains a unit vector. By an {\it orthonormal basis} $B$ of $\mathfrak H$, we mean a set $B$ of mutually orthogonal unit vectors of $\mathfrak H$ such that $B\c = \{0\}$. We note that if $\mathfrak H$ is finite-dimensional, $\mathfrak H$ is in this case the linear hull of $B$, that is, $B$ is an algebraic basis.

For $C \subseteq \mathfrak H$, we denote by $C\c$ the orthocomplement of $C$. A subspace of $\mathfrak H$ that is the orthocomplement of some $C \subseteq \mathfrak H$ is called {\it orthoclosed}. We note that every finite-dimensional subspace has this property. We write $\mathfrak H = \mathfrak M_1 \oplus \mathfrak M_2$ if $\mathfrak H$ is the direct sum of the orthogonal orthoclosed subspaces $\mathfrak M_1$ and $\mathfrak M_2$. A subspace $\mathfrak M$ is called {\it splitting} if $\mathfrak H = \mathfrak M \oplus {\mathfrak M}\c$. We refer to $\mathfrak H$ as an {\it orthomodular space} if every orthoclosed subspace is splitting.

Let $T \colon \mathfrak H_1 \to \mathfrak H_2$ be a linear map between Hermitian spaces. An {\it adjoint} of $T$ is a linear map $S \colon \mathfrak H_2 \to \mathfrak H_1$ such that $\herm{T(u)}{v} = \herm{u}{S(v)}$ for any $u \in \mathfrak H_1$ and $v \in \mathfrak H_2$. We call $T$ {\it adjointable} if $T$ possesses an adjoint, which in this case is unique and denoted by $T\adj$.

By an {\it isometry} between Hermitian spaces, we mean an adjointable map $U \colon \mathfrak H_1 \to \mathfrak H_2$ such that $U\adj \circ U = \id_{\mathfrak H_1}$. In other words, the linear map $U \colon \mathfrak H_1 \to \mathfrak H_2$ is an isometry if $U$ is adjointable and preserves the Hermitian form. Note that in this case $\mathfrak H_2 = \image U \oplus \kernel U\adj$. A bijective isometry is called a {\it unitary map}.

\begin{example}
1. Any $\star$-sfield $\mathbb F$ gives rise to the dagger category $\Her{\mathbb F}$ as follows. The objects of $\Her{\mathbb F}$ are the Hermitian spaces $\mathfrak H$ over $\mathbb F$, the morphisms are the adjointable linear maps between them, and the dagger is the adjoint.

2. Let $\mathbb F$ be one of the three classical $\star$-sfields, that is, $\Reals$, $\Complexes$, or $\Quaternions$ equipped with their respective canonical antiautomorphism. Let $\Hil{\mathbb F}$ be the dagger category whose objects are the Hilbert spaces over $\mathbb F$, whose morphisms are the bounded linear maps between them, and whose dagger is the usual adjoint.
\end{example}

Our aim is a characterisation of $\Hil{\Complexes}$. We will make use of its properties listed in the following proposition.

\begin{proposition} \label{prop:Hilbert-spaces-fulfil-axioms}
Let $\mathbb F$ be a classical $\star$-sfield. A map $T \colon \mathfrak H_1 \to \mathfrak H_2$ in $\Hil{\mathbb F}$ is a monomorphisms iff $T$ is injective, and an isomorphism iff $T$ is bijective. Moreover, $T$ is a dagger monomorphism iff $T$ is an isometry, and a dagger isomorphism iff $T$ is unitary.

Moreover, $\Hil{\mathbb F}$ has the following properties:
\begin{itemize}

\item[\rm (i)] The orthogonal sum of two Hilbert spaces in $\Hil{\mathbb F}$, together with the inclusion maps, is a dagger biproduct.

\item[\rm (ii)] Let $\mathbf D = \big(\mathfrak H_i \xlongrightarrow[]{T_{i \leq j}} \mathfrak H_j\big)_{i,j \in K}$ be a directed diagram of isometries in $\Hil{\mathbb F}$. Then $\mathbf D$ has in $(\Hil{\mathbb F})_\text{\rm dm}$ a colimit $\big(\mathfrak H_i \xlongrightarrow[]{U_i} \mathfrak H\big)_{i \in K}$ and the isometries $(U_i)_{i \in K}$ are jointly epic.

\item[\rm (iii)] Let $\mathfrak M$ be a closed subspace of a Hilbert space $\mathfrak H$. Then $\mathfrak H$ is the orthogonal sum of $\mathfrak M$ and ${\mathfrak M}\c$.

\item[\rm (iv)] Let $\mathfrak I \in \Hil{\mathbb F}$ be a one-dimensional Hilbert space. Then $\mathfrak I$ is dagger simple. Moreover, we have:

\begin{itemize}

\item[\rm (a)] For any non-zero Hilbert space $\mathfrak H$, there is a non-zero linear map $\mathfrak I \to \mathfrak H$.

\item[\rm (b)] For any non-zero linear map $T \colon \mathfrak I \to \mathfrak H$, there is an $\alpha \neq 0$ in the centre of $\mathbb F$ such that $\alpha T$ is an isometry.

\end{itemize}

\end{itemize}
If $\mathbb F = \Complexes$, we moreover have:

\begin{itemize}

\item[\rm (v)] Any unitary map $U \colon \mathfrak H \to \mathfrak H$ has a strict square root.

\end{itemize}
\end{proposition}

\begin{proof}
The initial assertions and parts (i), (iii), (iv) are clear.

Ad (ii): Let $\big(\mathfrak H_i \xlongrightarrow[]{U_i} {\mathfrak H}_\infty \big)_{i \in K}$ be the direct limit of the linear spaces $\mathfrak H_i$, $i \in K$; endow $\mathfrak H_\infty$ with the sesquilinear form of the $\mathfrak H_i$ via the maps $U_i$, $i \in K$; and let $\mathfrak H$ be the metric completion of $\mathfrak H_\infty$. Then it is readily seen that $\big(\mathfrak H_i \xlongrightarrow[]{U_i} {\mathfrak H}\big)_{i \in K}$ is a colimit of $\mathbf D$ in the wide subcategory of isometries. Moreover, if $S, T \colon \mathfrak H \to \mathfrak H'$ are morphisms of $\Hil{\mathbb F}$ such that $S \circ U_i = T \circ U_i$ for all $i$, then $S = T$ because $\mathfrak H_\infty$ is dense in $\mathfrak H$.

Ad (v): Assume that $\mathbb F = \Complexes$. The assertion follows from an application of Borel function calculus, see {\rm \cite[Lemma~4.5]{PaVe}}.
\end{proof}

\section{Dagger categories of Hilbert spaces}
\label{sec:main}

We present in this section a characterisation of $\Hil{\Complexes}$, the dagger category of complex Hilbert spaces and bounded linear maps.

Throughout this section, $\C$ will denote a dagger category with a zero object. We assume that $\C$ has the following properties:
\begin{itemize}

\item[\rm (H1)] $\C$ has dagger biproducts.

\item[\rm (H2)] The wide subcategory $\Cdm$ of dagger monomorphisms has directed colimits and the legs of any directed colimit cocone in $\Cdm$ are in $\C$ jointly epic.

\item[\rm (H3)] For every dagger monomorphism $f \colon A \to X$, there is a further dagger monomorphism $g \colon B \to X$ such that $\begin{tikzcd} A \arrow[r, "f"] & X & B \arrow[l, "g"'] \end{tikzcd}$ is a dagger biproduct.

\item[\rm (H4)]
$\C$ contains a dagger simple object. In addition, the following holds:

\begin{itemize}

\item[\rm (a)] For any dagger simple object $I$ and non-zero object $A$ in $C$, there is a non-zero morphism $u \colon I \to A$.

\item[\rm (b)] For any non-zero morphism $u \colon I \to A$ from a dagger simple object $I$ to an object $A \in \C$, there is an automorphism $h$ of $I$ such that $u \circ h$ is a dagger monomorphism.

\end{itemize}

\item[\rm (H5)] Every dagger automorphism has a strict square root.
\end{itemize}

Here, condition (H2) is understood in accordance with Proposition~\ref{prop:Hilbert-spaces-fulfil-axioms}(ii). Explicitly, let $\mathbf D = \big(A_i \xlongrightarrow[]{k_{i \leq j}} A_j\big)_{i,j \in K}$ be a directed diagram of dagger monomorphisms. According to (H2), $\mathbf D$ then has in $\Cdm$ a colimit $\big(A_i \xlongrightarrow[]{h_i} X\big)_{i \in K}$ and any two morphisms $f, g \colon X \to Y$ in $\C$ such that $f \circ h_i = g \circ h_i$ for all $i$ coincide.

By Proposition~\ref{prop:Hilbert-spaces-fulfil-axioms}, the axioms (H1)-(H5) are fulfilled by $\Hil{\Complexes}$. Our aim is to show that $\C$ is unitarily dagger equivalent to $\Hil{\Complexes}$.

\begin{lemma} \label{lem:l:unique-atomic-object}
\begin{itemize}

\item[\rm (i)] Every non-zero morphism between dagger simple objects is an isomorphism.

\item[\rm (ii)] $\C$ contains a unique dagger simple object up to dagger isomorphism.

\end{itemize}
\end{lemma}

\begin{proof}
Ad (i): Let $u \colon I \to J$ be a non-zero morphism between dagger simple objects. By (H4)(b), there is an automorphism $h \colon I \to I$ such that $u \circ h$ is a dagger monomorphism. As $J$ is dagger simple, $u \circ h$ is a dagger isomorphism. Hence $u$ is an isomorphism.

Ad (ii): Let $I$ and $J$ be dagger simple objects. By (H4)(a), there is non-zero morphism $u \colon I \to J$. As in part (i), 
there is an automorphism $h \colon I \to I$ such that $u \circ h$ is a dagger isomorphism. 
%
\end{proof}

In the sequel, $I$ will always denote a dagger simple object. We will endow the symbol $I$ with subscripts whenever the distinction of different copies of $I$ is necessary.

\begin{lemma} \label{lem:nI}
The objects $0$, $I$, and $I \oplus I$ are pairwise non-isomorphic.
\end{lemma}

\begin{proof}
By assumption, $0$ and $I$ are not isomorphic. Moreover, there is no dagger biproduct $\begin{tikzcd}[cramped] I \arrow[r, "f"] & 0 & I \arrow[l, "g"'] \end{tikzcd}$ and hence $I \oplus I$ is not isomorphic to $0$. Finally, assume that $\begin{tikzcd}[cramped] I \arrow[r, "f"] & I & I \arrow[l, "g"'] \end{tikzcd}$ is a dagger biproduct. Then $f$ and $g$ are non-zero and hence by Lemma~\ref{lem:l:unique-atomic-object} isomorphisms. But then also $g\adj \circ f$ is an isomorphism and hence non-zero, a contradiction. We conclude that $I \oplus I$ is not isomorphic to $I$.
\end{proof}

By Theorem \ref{thm:semiadditive-structure}, $\C$ possesses a unique semiadditive structure.

\begin{theorem} \label{thm:CII-is-field-K}
$({\C}(I,I); +, 0_{I,I}, \circ, \id_I)$ is a $\star$-sfield.
\end{theorem}

\begin{proof}
By Lemma~\ref{lem:l:unique-atomic-object}(i), every endomorphism of $I$ is either an automorphism or the zero map. It follows that ${\C}(I,I)$ is a division semiring. We shall show that there are non-zero $k_1, k_2 \in {\C}(I,I)$ such that $k_1 + k_2 = 0$. It will then be clear that ${\C}(I,I)$ has additive inverses, that is, ${\C}(I,I)$ is a division ring (cf.~\cite[4.34]{Gol}).

We consider the biproduct $\begin{tikzcd}[cramped] I \arrow[r, "\iota_1"] & I \oplus I & I \arrow[l, "\iota_2"'] \end{tikzcd}$. As $\Delta_I \neq 0$, there is by (H4)(b) an automorphism $h \colon I \to I$ such that $\Delta_I \circ h$ is a dagger monomorphism. By (H3), there is a dagger biproduct $\begin{tikzcd}[cramped] J \arrow[r, "f"] & I \oplus I & I \arrow[l, "\Delta_I \circ h"'] \end{tikzcd}$. By Lemma~\ref{lem:nI}, $\Delta_I \circ h$ is not an isomorphism and hence, by Lemma \ref{lem:zero-in-biproduct}, $J \neq 0$. By (H4)(a), there is hence a non-zero morphism $g \colon I \to J$. Let $k = g\adj \circ f\adj$. Then $k \colon I \oplus I \to I $ is a non-zero morphism such that $k \circ \Delta_I = 0$.

We put $k_1 = k \circ \iota_1$ and $k_2 = k \circ \iota_2$. Note that $k_1$ and $k_2$ cannot be both $0$.
\begin{equation} \label{fml:CII-is-field-K}
\begin{tikzcd}
& I    \\
I \arrow[ur, "k_1", shift left=0.5ex] \arrow[r, "\iota_1"] 
& I \oplus I  \arrow[u, dashed, "k"]
& I \arrow[l, "\iota_2"']\arrow[ul, "k_2"', shift right=0.5ex] \\
\end{tikzcd}
\quad\quad\quad
\begin{tikzcd}
& I \arrow[dl, "\id"', shift right=0.5ex] \arrow[d, dashed, "\Delta_I"] \arrow[dr, "\id", shift left=0.5ex] \\
I \arrow[d, "k_1"] \arrow[r, "\iota_1"', shift right=0.5ex]
& I \oplus I  \arrow[d, dashed, "k_1 \oplus k_2"]
\arrow[l, "{\iota_1}\adj"', shift right=0.5ex] \arrow[r, "{\iota_2}\adj", shift left=0.5ex]
& I  \arrow[l, "\iota_2", shift left=0.5ex]\arrow[d, "k_2"] \\
I \arrow[dr, "\id"', shift right=0.5ex] \arrow[r, "\iota_1"']
& I \oplus I \arrow[d, dashed, "\nabla_I"]
& I \arrow[dl, "\id", shift left=0.5ex] \arrow[l, "\iota_2"] \\
& I
\end{tikzcd}
\end{equation}
From the commutative diagram on the right side of (\ref{fml:CII-is-field-K}) we observe that
\begin{align*}
\nabla_I \circ (k_1\oplus k_2)\circ \iota_1
& \;=\; \nabla_I \circ \iota_1 \circ k_1 \;=\; \id_I \circ k_1 \;=\; k_1, \\
\nabla_I \circ (k_1\oplus k_2) \circ \iota_2
& \;=\; \nabla_I \circ \iota_2 \circ k_2 \;=\; \id_I \circ k_2  \;=\; k_2
\end{align*}
and hence $k = \nabla_I  \circ (k_1 \oplus k_2)$. We conclude that $k_1 + k_2 = \nabla_I \circ (k_1 \oplus k_2) \circ \Delta_I = k \circ \Delta_I = 0$. As one of $k_1$ and $k_2$ is non-zero, they are actually both non-zero.

In view of Lemma \ref{lem:dagger-of-sum}, we finally observe that $^\star$ is an involutive antiautomorphism of ${\C}(I,I)$.
\end{proof}

\begin{definition} \label{def:linear-structure-H}
Let $\mathbb F = \C(I,I)\op$. We denote the multiplication in $\mathbb F$ by $\cdot$, that is, we put $\alpha \cdot \beta = \beta \circ \alpha$.

Moreover, for $X \in \C$, let ${\mathcal V}(X) = \homset(I,X)$. Equip ${\mathcal V}(X)$ with the addition $+$ and the constant $0 = 0_{I,X}$. For $\alpha \in \mathbb F$ and $u \in {\mathcal V}(X)$, let $\alpha \cdot u = u \circ \alpha$. For $u, v \in {\mathcal V}(X)$, let
\[ \herm u v \;=\; v\adj \circ u. \]
Finally, for a morphism $f \colon X \to Y$ in $\C$, let
\begin{equation} \label{fml:Vf}
{\mathcal V}(f) \colon {\mathcal V}(X) \to {\mathcal V}(Y) \komma u \mapsto f \circ u.
\end{equation}
\end{definition}

\begin{theorem} \label{thm:Hermitian-structure-H}
Let $X \in \C$. Equipped with the addition, scalar multiplication, and the inner product according to Definition \ref{def:linear-structure-H}, ${\mathcal V}(X)$ is a Hermitian space over $\mathbb F$.

Moreover, for any morphism $f \colon X \to Y$, ${\mathcal V}(f)$ is an adjointable linear map from ${\mathcal V}(X)$ to ${\mathcal V}(Y)$. In fact, $\mathcal V$ is a dagger functor from $\mathcal C$ to $\Her{\mathbb F}$, which also preserves the addition.
\end{theorem}

\begin{proof}
By Theorem \ref{thm:CII-is-field-K}, $\mathbb F$ is a $\star$-sfield. We moreover know that $({\mathcal V}(X); +, 0)$ is a commutative monoid. For any $u \in {\mathcal V}(X)$, $(-1) \cdot u$ is an additive inverse of $u$, hence $({\mathcal V}(X); +, 0)$ is in fact an abelian group. It is finally evident that the multiplication $(\alpha, u) \mapsto \alpha \cdot u$ makes ${\mathcal V}(X)$ into a linear space.

We have to show that $\herm{\cdot}{\cdot}$ is a Hermitian form. Using the fact that composition distributes from both sides over the addition and that, by Lemma~\ref{lem:dagger-of-sum}, the addition is compatible with the dagger, we easily check that $\herm{\cdot}{\cdot}$ is sesquilinear. Clearly, $\herm{v}{u} = \herm{u}{v}\adj$ for $u, v \in {\mathcal V}(X)$.

Furthermore, $\herm{u}{u} = 0$ for some $u \in {\mathcal V}(X)$ means $u\adj \circ u = 0_{I,I}$. Assume that $u \neq 0$. By (H4)(b), there is an automorphism $h \colon I \to I$ such that $u \circ h$ is a dagger monomorphism. But then $\id_I = (u \circ h)\adj \circ u \circ h = 0_{I,I}$, a contradiction. This shows that $\herm{\cdot}{\cdot}$ is anisotropic.

For a morphism $f \colon X \to Y$, we readily check that ${\mathcal V}(f) \colon {\mathcal V}(X) \to {\mathcal V}(Y)$ is a linear map. The functoriality of $\mathcal V$ is immediate. Moreover, $\herm{{\mathcal V}(f)(u)}{v} = \herm{f \circ u}{v} = v\adj \circ f \circ u = \herm{u}{f\adj \circ v} = \herm{u}{{\mathcal V}(f\adj)(v)}$ and hence ${\mathcal V}(f\adj) = {\mathcal V}(f)\adj$. We conclude that $\mathcal V \colon \mathcal C \to \Her{\mathbb F}$ is a dagger functor.

Finally, let $f, g \colon X \to Y$. For $u \in {\mathcal V}(X)$, we have ${\mathcal V}(f+g)(u) = (f+g) \circ u = f \circ u + g \circ u = {\mathcal V}(f)(u) + {\mathcal V}(g)(u) = ({\mathcal V}(f) + {\mathcal V}(g))(u)$. Hence ${\mathcal V}(f+g) = {\mathcal V}(f) + {\mathcal V}(g)$.
\end{proof}

\begin{lemma} \label{lem:uniformity}
Let $X \in \C$.
\begin{itemize}

\item[\rm (i)] Any $u \in {\mathcal V}(X) \setminus \{0\}$ is a monomorphism. Moreover, $u$ is a unit vector if and only if $u$ is a dagger monomorphism.

\item[\rm (ii)] ${\mathcal V}(X)$ is uniform.

\item[\rm (iii)] ${\mathcal V}(X)$ has an orthonormal basis.
\end{itemize}
\end{lemma}

\begin{proof}
Ad (i): $u$ is a monomorphism by (H4)(b). Moreover, $\herm u u = 1$ if and only if $u\adj \circ u = \id_I$.

Ad (ii): Let $u \in {\mathcal V}(X) \setminus \{0\}$. By (H4)(b), there is an $\alpha \in \mathbb F$ such that $\alpha \cdot u = u \circ \alpha$ is a dagger monomorphism. By part (i), $\alpha \cdot u$ is a unit vector.

Ad (iii): By Zorn's Lemma, we may choose a maximal set $B = \{ e_\kappa \colon \kappa \in \Lambda \}$ consisting of pairwise orthogonal non-zero vectors in ${\mathcal V}(X)$. Then $B\c = \{0\}$ and the assertion follows from part (ii).
\end{proof}

\begin{lemma} \label{lem:dagger-mono}
Let $X \in \C$ and $I \xlongrightarrow[]{e_1} X, \ldots, I \xlongrightarrow[]{e_n} X$, where $n \geq 1$. Then $e_1, \ldots, e_n$ are pairwise orthogonal unit vectors of \/ ${\mathcal V}(X)$ if and only if
\[ [e_1, \ldots, e_n] \colon n \, I \to X \]
is a dagger monomorphism.
\end{lemma}

\begin{proof}
The case $n = 1$ is clear from Lemma~\ref{lem:uniformity}(i). We will show the assertion for $n = 2$; the case $n \geq 3$ works similarly.

Assume that $e_1, e_2$ are orthogonal unit vectors. By Lemma~\ref{lem:uniformity}(i), $e_1$ and $e_2$ are dagger monomorphisms. From the commutative diagram

\[ \begin{tikzcd}[column sep = 2.5cm]
I
& I \oplus I \arrow[l, "{\iota_1}\adj"'] \arrow[r, "{\iota_2}\adj"]
& I \\
& X \arrow[ul, "{e_1}\adj"] \arrow[u, dashed, "{({e_1}\adj,{e_2}\adj)}"' {yshift=3pt}] \arrow[ur, "{e_2}\adj"'] \\
I \arrow[uu, "\id"] \arrow[r, "\iota_1"'] \arrow[ur, "e_1"]
& I \oplus I \arrow[u, dashed, "{[e_1,e_2]}"' {yshift=-2pt}]
& I \arrow[uu, "\id"] \arrow[l, "\iota_2"]  \arrow[ul, "e_2"']
\end{tikzcd} \]
we observe that $({e_1}\adj,{e_2}\adj) = [e_1,e_2]\adj$ and $[e_1,e_2]\adj \circ [e_1,e_2] = ({e_1}\adj,{e_2}\adj) \circ [e_1,e_2] = \id_{I \oplus I}$. Hence $[e_1, e_2]$ is a dagger monomorphism.

Conversely, assume that $[e_1,e_2] \colon I \oplus I \to X$ is a dagger monomorphism. Then $e_1 = [e_1,e_2] \circ \iota_1$ and $e_2 = [e_1,e_2] \circ \iota_2$ imply ${e_1}\adj \circ e_2 = 0_{I,I} = 0$ and ${e_1}\adj \circ e_1 = {e_2}\adj \circ e_2 = \id_I = 1$.
\end{proof}

\begin{lemma} \label{lem:ONB}
Let $X \in \C$ and $I \xlongrightarrow[]{e_1} X, \ldots, I \xlongrightarrow[]{e_n} X$, where $n \geq 0$. Then $\{ e_1, \ldots, e_n \}$ is an orthonormal basis of \/ ${\mathcal V}(X)$ if and only if $X = n \, I$ via $e_1, \ldots, e_n$.

In this case, we have
\begin{equation} \label{fml:ONB}
u \;=\; ({e_1}\adj \circ u) \cdot e_1 \;+\; \ldots \;+\; ({e_n}\adj \circ u) \cdot e_n
\end{equation}
for any $u \in {\mathcal V}(X)$.
\end{lemma}

\begin{proof}
By (H4)(a), ${\mathcal V}(X)$ is the zero space if and only if $X = 0$. Hence the assertions are clear in the case $n = 0$. We shall verify the lemma in the case $n = 2$. The arguments for $n \geq 3$ and for $n = 1$ are similar.

($\Rightarrow$): Let $\{e_1,e_2\}$ be an orthonormal basis of ${\mathcal V}(X)$. Consider the dagger biproduct $\begin{tikzcd}[cramped] I \arrow[r, "\iota_1"] & I \oplus I & I \arrow[l, "\iota_2"'] \end{tikzcd}$. By Lemma~\ref{lem:dagger-mono}, $[e_1,e_2] \colon I \oplus I \to X$ is a dagger monomorphism. We claim that $[e_1, e_2]$ is in fact a dagger isomorphism. According to (H3), there is a biproduct $\begin{tikzcd}[cramped] I \oplus I \arrow[r, "{[e_1,e_2]}"] & X & R \arrow[l, "r"'] \end{tikzcd}$. If $R$ is non-zero, there is by (H4) a dagger monomorphism $u \colon I \to R$. But then $r \circ u$, which is by Lemma~\ref{lem:uniformity}(i) a unit vector, is orthogonal to $e_1$ and $e_2$, a contradiction. Thus, by Lemma~\ref{lem:zero-in-biproduct}, $[e_1,e_2]$ is indeed a dagger isomorphism. As $e_1 = [e_1,e_2] \circ \iota_1$ and $e_2 = [e_1,e_2] \circ \iota_2$ we conclude that $\begin{tikzcd}[cramped] I \arrow[r, "e_1"] & X & I \arrow[l, "e_2"'] \end{tikzcd}$ is a dagger biproduct.

($\Leftarrow$): Assume that $\begin{tikzcd}[cramped] I \arrow[r, "e_1"] & X & I \arrow[l, "e_2"'] \end{tikzcd}$ is a dagger biproduct. Then $e_1, e_2$ are orthogonal vectors of ${\mathcal V}(X)$ and by Lemma~\ref{lem:uniformity}(i), $e_1, e_2$ are unit vectors. For any $u \colon I \to X$, we moreover have $u = e_1 \circ {e_1}\adj \circ u + e_2 \circ {e_2}\adj \circ u = ({e_1}\adj \circ u) \cdot e_1 + ({e_2}\adj \circ u) \cdot e_2$ by Lemma~\ref{lem:plus-oplus-2}. This shows (\ref{fml:ONB}) and we conclude that $\{e_1,e_2\}$ is an orthonormal basis.
\end{proof}

We will say that an object $X \in \C$ {\it has rank $n$}, where $n \in \Naturals$, if ${\mathcal V}(X)$ is $n$-dimensional, and we will say that $X$ has {\it infinite rank} if ${\mathcal V}(X)$ is infinite-dimensional.

By Lem\-ma~\ref{lem:ONB}, $X$ has rank $n$ if and only if $X$ is dagger isomorphic to $n \, I$. In particular, there is an object in $\C$ of any given finite rank. The next lemma shows that, in a sense, $\C$ contains objects of arbitrary size.

\begin{lemma} \label{lem:gen-colimit-of-finite-dimensional-subspaces}
Let $E$ be a non-empty set and let  $\mathcal F$ be the set of all non-empty finite subsets of $E$, partially ordered by inclusion.  For every $F = \{ \kappa_1, \ldots, \kappa_k \} \in {\mathcal F}$, let $A_F = I_{\kappa_1} \oplus \ldots \oplus I_{\kappa_k}$ via inclusion maps $\iota_{F,\kappa_1}, \ldots, \iota_{F,\kappa_k}$.  For every $F, G \in {\mathcal F}$ such that $F \subseteq G$ and $F = \{ \kappa_1, \ldots, \kappa_k \}$, let $k_{F \subseteq G} = [\iota_{G,{\kappa_1}}, \ldots \iota_{G,{\kappa_k}}] \colon A_F \to A_G$. In $\Cdm$, $\big(A_F \xlongrightarrow[]{k_{F \subseteq G}} A_G\big)_{F,G \in {\mathcal F}}$ is then a diagram of type $\mathcal F$ and has a colimit $\big(A_F \xlongrightarrow[]{h_F} X\big)_{F \in {\mathcal F}}$. Moreover, $\{ h_{\{ \kappa \}} \colon \kappa \in E\}$ is an orthonormal basis of ${\mathcal V}(X)$.
\end{lemma}

\begin{proof}
%
For $F, G \in {\mathcal F}$ such that $F \subseteq G$, $k_{F \subseteq G}$ is by Lemma~\ref{lem:dagger-mono} a dagger monomorphism. By pre-composing with the inclusion maps, we moreover readily check that $\mathbf D = \big(A_F \xlongrightarrow[]{k_{F \subseteq G}} A_G\big)_{F,G \in {\mathcal F}}$ is a diagram of type $\mathcal F$.

By (H2), $\mathbf D$ has in $\Cdm$ a colimit $\big(A_F \xlongrightarrow[]{h_F} X\big)_{F \in {\mathcal F}}$. It remains to show that $\{ A_{\{\kappa\}} \xlongrightarrow[]{h_{\{\kappa\}}} X \colon \kappa \in E \}$ is an orthonormal basis of ${\mathcal V}(X)$. For each $\kappa \in E$, $h_{\{\kappa\}}$ is a dagger monomorphism and hence a unit vector of ${\mathcal V}(X)$. Moreover, for distinct $\kappa,\kappa' \in E$, $h_{\{\kappa\}} = h_{\{\kappa,\kappa'\}} \circ k_{\{\kappa\},\{\kappa,\kappa'\}}$ is orthogonal to $h_{\{\kappa'\}} = h_{\{\kappa,\kappa'\}} \circ k_{\{\kappa'\},\{\kappa,\kappa'\}}$. Assume that $u \in {\mathcal V}(X)$ is orthogonal to every $h_{\{\kappa\}}$, $\kappa \in E$. This means $u\adj \circ h_{\{\kappa\}} = 0_{A_{\{\kappa\},I}}$ for any $\kappa \in E$. It follows $u\adj \circ h_F = 0_{A_F,I}$ for any $F \in {\mathcal F}$. By the second part of (H2), $(h_F)_{F \in {\mathcal F}}$ are jointly epic. Hence $u\adj = 0_{X,I}$, that is, $u = 0$.
\end{proof}

\begin{lemma} \label{lem:infinite-rank-object}
Let $\lambda$ be a cardinal. Then there is an $X \in \C$ such that ${\mathcal V}(X)$ has an orthonormal basis of cardinality $\lambda$.
\end{lemma} 

\begin{proof}
This is immediate from Lemma~\ref{lem:gen-colimit-of-finite-dimensional-subspaces}.
\end{proof}

We may now formulate an analogue of Lemma~\ref{lem:ONB} for arbitrary objects.

\begin{lemma} \label{lem:colimit-of-finite-dimensional-subspaces}
Let $X \in \C$ be non-zero and let $E$ be a non-empty set of pairwise orthogonal unit vectors of ${\mathcal V}(X)$. Let $\mathcal F$ be the set of all non-empty finite subsets of $E$, partially ordered by inclusion. For each $F = \{ I_{e_1} \xlongrightarrow[]{e_1} X, \; \ldots, \; I_{e_n} \xlongrightarrow[]{e_n} X \}\in \mathcal F$, let $A_F = I_{e_1} \oplus \ldots \oplus I_{e_n}$ and $g_F = [e_1, \ldots, e_n] \colon A_F \to X$. Then there is a unique diagram $\mathbf D$ of type $\mathcal F$ such that $\big(A_F \xlongrightarrow[]{g_F} X\big)_{F \in {\mathcal F}}$ is a cocone of $\mathbf D$.

Moreover, $E$ is an orthonormal basis of ${\mathcal V}(X)$ if and only if $\big(A_F \xlongrightarrow[]{g_F} X\big)_{F \in {\mathcal F}}$ is in $\Cdm$ a colimit of $\mathbf D$. In particular, $X$ is in $\Cdm$ the colimit of a directed diagram of objects of finite rank.
\end{lemma}

\begin{proof}
For $F, G \in {\mathcal F}$ such that $F \subseteq G$, we define the dagger monomorphisms $k_{F \subseteq G}$ as in Lemma~\ref{lem:gen-colimit-of-finite-dimensional-subspaces}. According to this lemma, $\mathbf D = \big(A_F \xlongrightarrow[]{k_{F \subseteq G}} A_G\big)_{F,G \in {\mathcal F}}$ is a diagram of type $\mathcal F$ in $\Cdm$. Moreover, for every $F, G \in {\mathcal F}$ such that $F \subseteq G$, the pre-composition with the inclusion maps shows $g_G \circ k_{F \subseteq G} = g_F$. We conclude that $\big(A_F \xlongrightarrow[]{g_F} X\big)_{F \in {\mathcal F}}$ is a cocone of $\mathbf D$. Moreover, as $g_G$ is a dagger monomorphism, $k_{F \subseteq G}$ is the unique morphism such that $g_G \circ k_{F \subseteq G} = g_F$. That is, $\mathbf D$ is the unique diagram of type $\mathcal F$ whose cocone is $\big(A_F \xlongrightarrow[]{g_F} X\big)_{F \in {\mathcal F}}$.

Assume now that $E$ is an orthonormal basis of ${\mathcal V}(X)$. Let $\big(A_F \xlongrightarrow[]{h_F} A\big)_{F \in {\mathcal F}}$ be the colimit of $\mathbf D$ in $\Cdm$ and let $g \colon A \to X$ be the unique dagger monomorphism associated with $\big(A_F \xlongrightarrow[]{g_F} X\big)_{F \in {\mathcal F}}$. By (H3), there is a biproduct $\begin{tikzcd}[cramped] A \arrow[r, "g"] & X & R \arrow[l, "r"'] \end{tikzcd}$. If $R$ is non-zero, there is a  dagger monomorphism $u \colon I \to R$. As any $e \in E$ factorises through $A$, this means that $r \circ u$ is a unit vector orthogonal to $E$, a contradiction. Hence $g$ is a dagger isomorphism and we have shown that, in $\Cdm$, $\big(A_F \xlongrightarrow[]{g_F} X\big)_{F \in {\mathcal F}}$ is a colimit of $\mathbf D$.

Conversely, assume that $\big(A_F \xlongrightarrow[]{g_F} X\big)_{F \in {\mathcal F}}$ is in $\Cdm$ a colimit of $\mathbf D$. Let $u \in {\mathcal V}(X)$ be orthogonal to every $e \in E$. This means $u\adj \circ e = 0_{I,I}$ for any $e \in E$ and it follows that $u\adj \circ g_F = 0_{A_F,I}$ for any $F \in {\mathcal F}$. By (H2), $u\adj = 0_{X,I}$ and thus $u = 0$. This shows that $E$ is an orthonormal basis.
\end{proof}

\begin{proposition} \label{prop:V-is-faithful}
The dagger functor $\mathcal V$ is faithful.
\end{proposition}

\begin{proof}
Let $f, g \colon X \to Y$ be such that ${\mathcal V}(f) = {\mathcal V}(g)$. This means $f \circ u = g \circ u$ for any $u \colon I \to X$. Hence $f \circ h = g \circ h$ for any dagger monomorphism $h$ from an object of finite rank to $X$. By Lemma~\ref{lem:colimit-of-finite-dimensional-subspaces}, $X$ is in $\Cdm$ the colimit of a directed diagram of objects of finite rank. By the second part of (H2), we conclude $f = g$.
\end{proof}

Our next aim is to establish, for any $X \in \C$, a correspondence between the orthoclosed subspaces of ${\mathcal V}(X)$ and the dagger monomorphisms into $X$.

\begin{lemma} \label{lem:dagger-mono-induces-isometry}
Let $h \colon A \to X$ be a dagger monomorphism.
\begin{itemize}

\item[\rm (i)] ${\mathcal V}(h)$ is an isometry. In particular, $\mathfrak M = \image {\mathcal V}(h)$ is an orthoclosed subspace of ${\mathcal V}(X)$, its orthocomplement is ${\mathfrak M}\c = \kernel {\mathcal V}(h\adj) = \{ u \in {\mathcal V}(X) \colon h\adj \circ u = 0 \}$, and we have ${\mathcal V}(X) = \mathfrak M \oplus {\mathfrak M}\c$.

\item[\rm (ii)] Let $\begin{tikzcd}[cramped] A \arrow[r, "h"] & X & B \arrow[l, "k"'] \end{tikzcd}$ be a dagger biproduct. Then ${\mathcal V}(h)$ and ${\mathcal V}(k)$ are isometries, and ${\mathcal V}(X) = \image {\mathcal V}(h) \oplus \image {\mathcal V}(k)$.

\item[\rm (iii)] If $h$ is an isomorphism, then ${\mathcal V}(h)$ is a unitary map.

\end{itemize}
\end{lemma}

\begin{proof}
Ad (i): By functoriality, ${\mathcal V}(h)\adj \circ {\mathcal V}(h) = {\mathcal V}(h\adj) \circ {\mathcal V}(h) = {\mathcal V}(h\adj \circ h) = {\mathcal V}(\id_A) = \id_{{\mathcal V}(A)}$. Hence ${\mathcal V}(h)$ is an isometry and it further follows that ${\mathcal V}(X) = \image {\mathcal V}(h) \oplus \kernel {\mathcal V}(h)\adj = \image {\mathcal V}(h) \oplus \kernel {\mathcal V}(h\adj)$.

Ad (ii): By part (i), ${\mathcal V}(h)$ and ${\mathcal V}(k)$ are isometries. Moreover, we clearly have $\image {\mathcal V}(h) \perp \image {\mathcal V}(k)$, and for $u \in {\mathcal V}(X)$,
\[ u \;=\; h \circ h\adj \circ u + k \circ k\adj \circ u
\;=\; {\mathcal V}(h)(h\adj \circ u) + {\mathcal V}(k)(k\adj \circ u) \]
by Lemma~\ref{lem:plus-oplus-2}, hence ${\mathcal V}(X) = \image {\mathcal V}(h) + \image {\mathcal V}(k)$.

Ad (iii): As in part (i), we conclude ${\mathcal V}(h)\adj \circ {\mathcal V}(h) = \id_{{\mathcal V}(A)}$ and ${\mathcal V}(h) \circ {\mathcal V}(h)\adj = \id_{{\mathcal V}(X)}$.
\end{proof}

By Lemmas~\ref{lem:dagger-mono-induces-isometry}(i), any dagger monomorphism $h \colon A \to X$ gives rise to the orthoclosed subspace $\image {\mathcal V}(h) = \{ h \circ u \colon u \in {\mathcal V}(A) \}$ of ${\mathcal V}(X)$. We now show that every orthoclosed subspace of ${\mathcal V}(X)$ arises in this way.

\begin{lemma} \label{lem:orthoclosed-subspaces-dagger-mono}
Let $X \in \C$ and let $\mathfrak M$ be an orthoclosed subspace of \/ ${\mathcal V}(X)$. Then there is a dagger monomorphism $h$ such that $\mathfrak M = \image {\mathcal V}(h)$.
\end{lemma}

\begin{proof}
If $\mathfrak M$ is $0$-dimensional the zero map $h \colon 0 \to X$ will do. We assume that $\mathfrak M$ is at least $1$-dimensional.

Let $\mathfrak M = C\c$ for some $C \subseteq {\mathcal V}(X)$ and
let $E$ be an orthonormal basis of $\mathfrak M$. Let $\mathcal F$ be the set of non-empty finite subsets of $E$ and define, according to Lemma~\ref{lem:colimit-of-finite-dimensional-subspaces}, the diagram $\mathbf D = \big(A_F \xlongrightarrow[]{k_{F \subseteq G}} A_G\big)_{F,G \in {\mathcal F}}$ with the cocone $\big(A_F \xlongrightarrow[]{g_F} X\big)_{F \in {\mathcal F}}$. By (H2), $\mathbf D$ has a colimit $\big(A_F \xlongrightarrow[]{h_F} A\big)_{F \in {\mathcal F}}$ in $\Cdm$. Let $h \colon A \to X$ be the associated unique dagger monomorphism.

We claim that $\mathfrak M = \image {\mathcal V}(h)$. If $v \in C$, then $v\adj \circ e = 0_{I,I}$ for any $e \in E$, hence $v\adj \circ g_F = v\adj \circ h \circ h_F = 0_{A_F,I}$ for any $F \in {\mathcal F}$. By (H2), $v\adj \circ h = 0$. In particular, $h \circ u \perp C$ for any $u \in {\mathcal V}(A)$, that is, $\image {\mathcal V}(h) \subseteq C\c = \mathfrak M$. To show the reverse inclusion, let $u \in \mathfrak M$. By Lemma~\ref{lem:dagger-mono-induces-isometry}(i), $\image {\mathcal V}(h)$ is a splitting subspace of ${\mathcal V}(X)$. Consequently, $u = v + w$, where $v \in \image {\mathcal V}(h)$ and $w \perp \image {\mathcal V}(h)$. As $v \in \mathfrak M$, we have $w \in \mathfrak M$; and as $E \subseteq \image {\mathcal V}(h)$, we have $w \in E\c$. It follows $w = 0$ and $u \in \image {\mathcal V}(h)$.
\end{proof}

As a consequence, we see that ${\mathcal V}(X)$ is for any $X \in \C$ an orthomodular space.

\begin{lemma} \label{lem:orthomodularity-of-objects}
Let $X \in \C$ and let $\mathfrak M$ be an orthoclosed subspace of
\/ ${\mathcal V}(X)$. Then there is a biproduct $\begin{tikzcd}[cramped] A \arrow[r, "h_A"] & X & B \arrow[l, "h_B"'] \end{tikzcd}$ such that $\mathfrak M = \image {\mathcal V}(h_A)$ and ${\mathfrak M}\c = \image {\mathcal V}(h_B)$. We have ${\mathcal V}(X) = \mathfrak M \oplus {\mathfrak M}\c$, and ${\mathcal V}(h_A \circ {h_A}\adj)$ is the orthogonal projection onto $\mathfrak M$.

In particular, ${\mathcal V}(X)$ is an orthomodular space.
\end{lemma}

\begin{proof}
By Lemma~\ref{lem:orthoclosed-subspaces-dagger-mono}, there is a dagger monomorphism $h_A \colon A \to X$ such that $\mathfrak M = \image {\mathcal V}(h_A)$. By Lemma~\ref{lem:dagger-mono-induces-isometry}(i), we have ${\mathcal V}(X) = \mathfrak M \oplus {\mathfrak M}\c$. By (H3), there is a biproduct $\begin{tikzcd}[cramped] A \arrow[r, "h_A"] & X & B \arrow[l, "h_B"'] \end{tikzcd}$. By Lemma~\ref{lem:dagger-mono-induces-isometry}(ii), ${\mathfrak M}\c = \image {\mathcal V}(h_B)$.

Let $P = {\mathcal V}(h_A \circ {h_A}\adj)$. For any $u \in \mathfrak M$ there is a $v \in {\mathcal V}(A)$ such that $u = h_A \circ v$ and consequently ${\mathcal V}(P)(u) = h_A \circ {h_A}\adj \circ h_A \circ v = h_A \circ v = u$. Similarly, for any $u \perp \mathfrak M$, there is a $v \in {\mathcal V}(B)$ such that $u = h_B \circ v$ and consequently ${\mathcal V}(P)(u) = h_A \circ {h_A}\adj \circ h_B \circ v = 0$. This shows that $P$ is the projection onto $\mathfrak M$.
\end{proof}

\begin{theorem}
For any $X \in \C$,  ${\mathcal V}(X)$ is a Hilbert space over $\Complexes$.
\end{theorem}

\begin{proof}
Let $X \in \C$ be such that ${\mathcal V}(X)$ is infinite-dimensional. By Lemma~\ref{lem:infinite-rank-object}, there is such an object. Then, by Lemmas~\ref{lem:uniformity} and~\ref{lem:orthomodularity-of-objects}, ${\mathcal V}(X)$ is a uniform orthomodular space over $\mathbb F$. Hence, by Sol\` er's Theorem \cite{Sol}, $\mathbb F$ is a classical $\star$-sfield and ${\mathcal V}(X)$ is a Hilbert space. Let $X \in \C$ now have finite rank. Then ${\mathcal V}(X)$ is a finite-dimensional uniform Hermitian space over $\mathbb F$ and hence likewise a Hilbert space.

Consider now the biproduct $\begin{tikzcd}[cramped] I \arrow[r, "e_1"] & I \oplus I & I \arrow[l, "e_2"'] \end{tikzcd}$ and let $\mathfrak H = {\mathcal V}(I \oplus I)$. By Lemma~\ref{lem:ONB}, $\{e_1,e_2\}$ is an orthonormal basis of $\mathfrak H$. Let $h = (-1) \oplus (-1)$ and $U = {\mathcal V}(h)$. Then $U(e_1) = ((-1) \oplus (-1)) \circ e_1 = e_1 \circ (-1) = (-1) \cdot e_1 = -e_1$ and similarly $U(e_2) = -e_2$. As $-1$ is in the centre of $\mathbb F$, we conclude $U = -\id_{\mathfrak H}$.

$h$ is a dagger automorphism and hence, by (H5), there is a strict square root $r$ of $h$. Let $V = {\mathcal V}(r)$. Then $V^2 = {\mathcal V}(r^2) = {\mathcal V}(h) = U$. Let $\mathfrak M$ be a subspace of $\mathfrak H$, let $P_{\mathfrak M}$ be the projection onto it, and, in accordance with Lemma~\ref{lem:orthomodularity-of-objects}, let $p_{\mathfrak M}$ be the endomorphism of $I \oplus I$ such that $P_{\mathfrak M} = {\mathcal V}(p_{\mathfrak M})$. Then
\[ {\mathcal V}(p_{\mathfrak M} \circ h)
\;=\; P_{\mathfrak M} \circ U
\;=\; -P_{\mathfrak M}
\;=\; U \circ P_{\mathfrak M}
\;=\; {\mathcal V}(h \circ p_{\mathfrak M}) \]
and by faithfulness $p_{\mathfrak M} \circ h = h \circ p_{\mathfrak M}$. It follows $p_{\mathfrak M} \circ r = r \circ p_{\mathfrak M}$ by assumption and hence $P_{\mathfrak M} \circ V = V \circ P_{\mathfrak M}$. In particular, each $1$-dimensional subspace is invariant for $V$ and we conclude that $V = \alpha \, \id_{\mathfrak H}$ for some $\alpha \in \mathbb F$. As $V$ is a linear map, $\alpha$ must belong to the centre of $\mathbb F$. We have that $\alpha^2 = -1$ and this is possible only if $\mathbb F = \Complexes$.
\end{proof}

\begin{proposition} \label{prop:V-is-essentially-surjective}
$\mathcal V$ is dagger essentially surjective.
\end{proposition}

\begin{proof}
By Lemma~\ref{lem:infinite-rank-object}, there is for any cardinal number $\lambda$ an $X \in \C$ such that ${\mathcal V}(X)$ has (Hilbert) dimension $\lambda$.
\end{proof}

\begin{lemma} \label{lem:dagger-mono-between-any-pair}
For any objects $X, Y \in \C$, there is a dagger monomorphism from $X$ to $Y$ or from $Y$ to $X$.
\end{lemma}

\begin{proof}
Let ${\mathcal V}(X)$ have the orthonormal basis $\{ I_\kappa \xlongrightarrow[]{e_\kappa} X \;\colon\; \kappa \in \Lambda_X \}$, let ${\mathcal V}(Y)$ have the orthonormal basis $\{ I_\kappa \xlongrightarrow[]{f_\kappa} Y \;\colon\; \kappa \in \Lambda_Y \}$, and assume that $\Lambda_X \subseteq \Lambda_Y$.

Let $\mathcal F$ be the set of non-empty finite subsets of $\Lambda_X$, partially ordered by inclusion. For $F = \{ \kappa_1, \ldots, \kappa_k \} \in {\mathcal F}$, let $A_F = I_{\kappa_1} \oplus \ldots \oplus I_{\kappa_k}$ and $g_F = [e_{\kappa_1}, \ldots, e_{\kappa_k}]$. In accordance with Lemma~\ref{lem:colimit-of-finite-dimensional-subspaces}, let $\mathbf D$ be the diagram that has in $\Cdm$ the colimit $\big(A_F \xlongrightarrow[]{g_F} X\big)_{F \in {\mathcal F}}$.

For $F = \{ \kappa_1, \ldots, \kappa_k \} \in {\mathcal F}$, let moreover $m_F = [f_{\kappa_1}, \ldots, f_{\kappa_k}] \colon A_F \to Y$. Then $\big(A_F \xlongrightarrow[]{m_F} Y\big)_{F \in {\mathcal F}}$ is a cocone of $\mathbf D$ in $\Cdm$. We conclude that there is a dagger monomorphism $X \to Y$.
\end{proof}

In order to establish fullness we have to show that $\mathcal V$ captures all bounded operators. For Hilbert spaces of dimension at least $2$, a theorem of Airat M.~Bikchentaev provides exactly what we need to this end. Indeed, we already know that all projections are in the image of $\mathcal V$. By Bikchentaev's Theorem, every bounded operator is a sum of products of projections.

\begin{theorem} \label{thm:Bikchentaev}
Let $\mathfrak H$ be the complex Hilbert space of dimension $\geq 2$. Then the ring of bounded operators of $\mathfrak H$ is generated by the projections.
\end{theorem}

\begin{proof}
See \cite[Section 3, Theorem]{Bik}.
\end{proof}

\begin{proposition} \label{prop:V-is-full}
$\mathcal V$ is full.
\end{proposition}

\begin{proof}
For any $\mathfrak H$ in the image of $\mathcal V$ such that $\dim \mathfrak H \geq 2$, let $\E(\mathfrak H)$ be the set of all endomorphisms of $\mathfrak H$ in the image of $\mathcal V$. As $\mathcal V$ preserves the adjoint, each $\E(\mathfrak H)$ consists of bounded maps only. By Lemma~\ref{lem:orthomodularity-of-objects}, $\E(\mathfrak H)$ contains all projections; as $\mathcal V$ preserves $+$, $\E(\mathfrak H)$ is closed under addition; and $\E(\mathfrak H)$ is closed under composition. Hence, by Theorem~\ref{thm:Bikchentaev}, $\E(\mathfrak H)$ consists of all bounded endomorphisms of $\mathfrak H$.

Let now $\mathfrak H_1$ and $\mathfrak H_2$ arbitrary Hilbert spaces in the image of $\mathcal V$ such that $\dim \mathfrak H_1 \leq \dim \mathfrak H_2$. Then there is by Lemmas~\ref{lem:dagger-mono-between-any-pair} and~\ref{lem:dagger-mono-induces-isometry}(i) an isometry from $\mathfrak H_1$ to $\mathfrak H_2$ in the image of $\mathcal V$. It readily follows that the image of $\mathcal V$ contains all bounded maps between any two spaces.
\end{proof}

\begin{theorem}
$\C$ is unitarily dagger equivalent to $\Hil{\Complexes}$.
\end{theorem}

\begin{proof}
By Propositions~\ref{prop:V-is-faithful},~\ref{prop:V-is-full}, and~\ref{prop:V-is-essentially-surjective}, ${\mathcal V} \colon \C \to \Hil{\Complexes}$ is faithful, full, and dagger essentially surjective. Hence the assertion follows from \cite[Lemma~V.1]{Vic}.
\end{proof}

{\bf Acknowledgements.} This research was funded in part by the Austrian Science Fund (FWF) 10.55776/PIN5424624 and the Czech Science Foundation (GA\v CR) 25-20013L.

\end{document}